\documentclass[12pt]{amsart}
\usepackage[active]{srcltx}
\usepackage[centertags]{amsmath}
\usepackage{latexsym}
\usepackage{amssymb}

\DeclareMathOperator{\Tr}{Tr}

\newcommand{\F}{\mathbb F}
\newcommand{\barF}{\overline{\mathbb F}}
\newcommand{\barFq}{\overline{{\mathbb F}_q}}
\newcommand{\Proj}{\mathbb P}

\newcommand{\hide}[1]{}
\newtheorem{dummy}{Dummy}

\numberwithin{equation}{section}

\newtheorem{lemma}[dummy]{Lemma}

\newtheorem{theorem}[dummy]{Theorem}

\newtheorem{cor}[dummy]{Corollary}

\theoremstyle{definition}

\theoremstyle{remark}
\newtheorem{assumptions}[dummy]{Assumptions}
\newtheorem{property}[dummy]{Property}
\newtheorem{rem}[dummy]{Remark}

\begin{document}

\bibliographystyle{amsalpha}
\author{Sandro Mattarei}
\email{mattarei@science.unitn.it}
\address{Dipartimento di Matematica\\
  Universit\`a degli Studi di Trento\\
  via Sommarive 14\\
  I-38123 Povo (Trento)\\
  Italy}
\title{Inversion and subspaces of a finite field}
\date{\today}
\begin{abstract}
Consider two $\F_q$-subspaces $A$ and $B$ of a finite field, of the same size,
and let $A^{-1}$ denote the set of inverses of the nonzero elements of $A$.
The author proved that $A^{-1}$ can only be contained in $A$ if either $A$ is a subfield,
or $A$ is the set of trace zero elements in a quadratic extension of a field.
Csajb\'{o}k refined this to the following quantitative statement:
if $A^{-1}\not\subseteq B$, then the bound
$|A^{-1}\cap B|\le 2|B|/q-2$ holds.
He also gave examples showing that his bound is sharp for $|B|\le q^3$.
Our main result is a proof of the stronger bound
$|A^{-1}\cap B|\le |B|/q\cdot\bigl(1+O_d(q^{-1/2})\bigr)$,
for $|B|=q^d$ with $d>3$.
We also classify all examples with $|B|\le q^3$
which attain equality or near-equality in Csajb\'{o}k's bound.
\end{abstract}
\maketitle

\section{Introduction}

In response to a question of Andrea Caranti, for use in~\cite{CDVS:AES}, the author determined in~\cite{Mat:inverse-closed}
the additive subgroups of a field which are closed with respect to inverting nonzero elements.
The more general question with a division ring instead of a field was independently answered
in~\cite{GGSZ}.
The proofs depend on Hua's identity~\cite{Hua:sfield_properties}, and on Jordan algebra
techniques to cover the noncommutative case.
However, a more direct argument based on polynomials was given in~\cite{Mat:inverse-closed}
in the special case of finite fields, which appears to have
attracted some attention for cryptographic applications.
In that special case the result reads as follows:
a non-trivial inverse-closed additive subgroup $A$ of a finite field $E$ is either a subfield of $E$ or the set of elements
of trace zero in some quadratic field extension contained in $E$.

In~\cite{Csajbok:inverse-closed}, Bence Csajb\'{o}k investigated a
question which may be thought of as a refinement of this result:
can one obtain the same conclusion from the weaker assumption that $A$ is an additive subgroup of a finite field which is
{\em almost} inverse-closed, in the sense that {\em most} of the inverses
of its nonzero elements belong to $A$?
Of course the two words in italics need to be given a precise meaning.
A very special case of this occurred in~\cite[Lemma~5.3]{KLS},
where the conclusion was proved under the assumption that $A$ is
inverse closed up to at most two nonzero elements.

It turns out that this question is better studied in the more
general form where two additive subgroups $A$ and $B$ of the same size of a finite
field are considered, and one asks for an upper bound on $|A^{-1}\cap B|$
in terms of $|B|$ in case $A^{-1}\not\subseteq B$.
Here $S^{-1}$, for $S$ a subset of a field, denotes the set of inverses of the nonzero elements of $S$.
Note that the intersection $A^{-1}\cap B$ attains maximal size $q^d-1$ exactly when $A^{-1}\subseteq B$,
and in that case $A$ and $B$ are both (one-dimensional) $\F_{q^d}$-subspaces.
Because the ambient finite field plays only a minor role, it appears convenient
to work in the algebraic closure of a finite field, and so we
rather state Csajb\'{o}k's results in the following equivalent form.

\begin{theorem}[Theorems~1.2 and~3.1 in~\cite{Csajbok:inverse-closed}]\label{thm:bound}
Let $A$ and $B$ be finite nonzero $\F_q$-subspaces of $\barFq$ of the same size,
with $A^{-1}\not\subseteq B$.
Then
$|A^{-1}\cap B|\le 2|B|/q-2$.

When $q=2$ and $|B|>2$ the conclusion can be strengthened to
$|A^{-1}\cap B|\le 3|B|/4-1$.
\end{theorem}

In Section~\ref{sec:meet} we present a proof of Theorem~\ref{thm:bound}
which is shorter than Csajb\'{o}k's original proof, but also more explicit.
This is because, say in case of the former bound of Theorem~\ref{thm:bound}, our proof produces
a polynomial $C(x)$, of degree $2|B|/q-2$, explicitly computable from the polynomials defining $A$ and $B$, whose set of roots contains
$A^{-1}\cap B$.
This can then be effectively used for further study of
$A^{-1}\cap B$, as we illustrate next.

A natural question which arises at this point is whether the
bounds given in Theorem~\ref{thm:bound} are best possible, especially the
general bound which holds for arbitrary $q$.
In the early draft of~\cite{Csajbok:inverse-closed} which was
available to the author during most of the writing of this paper, the few examples provided fell short of
showing sharpness of the bound beyond the rather trivial cases where $|B|\le q^2$,
which we briefly discuss at the end of Section~\ref{sec:meet}.

The present work begun as an attempt to provide
examples with $|B|=q^3$ where equality is attained in Csajb\'{o}k's bound.
We present such examples in Section~3.
They appear in Theorem~\ref{thm:q^3_full}, within a more general situation where
$(A^{-1}\cap B)\cup\{0\}$ contains a one-dimensional $\F_{q^2}$-subspace of $\barFq$.
In fact, under this assumption, which we will later show not to be restrictive, our proof of Theorem~\ref{thm:bound}
is especially effective: it allows us to give a polynomial description of
all pairs $(A,B)$ of three-dimensional spaces which attain equality in Csajb\'{o}k's bound,
that is, which satisfy $|A^{-1}\cap B|=2q^2-2$.
This is possible only for $q$ odd and, in geometric language, it occurs exactly when the image of $A^{-1}\cap B$ in the projective plane $\Proj B\cong\Proj^{2}(\F_q)$
associated with the linear $\F_q$-space $B$ is the union of a (nondegenerate) conic and an external line.
The configuration of the union of a conic and a secant line also arises,
and the corresponding subspaces then satisfy $|A^{-1}\cap B|=2q^2-2q$.
In even characteristic those two configurations collapse to that of
a conic and a tangent line, whence $|A^{-1}\cap B|=2q^2-q-1$.
The final, published version of~\cite{Csajbok:inverse-closed},
includes a study of the case where $A$ and $B$ have dimension three,
providing a presentation of the examples which we have briefly described.
However, Csajb\'{o}k's geometric approach limits his results to subspaces of $\F_{q^4}$ for $q$ odd,
and there is little overlap with our results, see Remark~\ref{rem:comparison}.

The claim we implicitly made above, that we have actually found
all pairs $(A,B)$ of three-dimensional $\F_q$-subspaces of $\barFq$ which attain equality in Csajb\'{o}k's bound,
relies on removing our additional assumption that
$(A^{-1}\cap B)\cup\{0\}$ contains a one-dimensional $\F_{q^2}$-subspace of $\barFq$.
We do that in Section~\ref{sec:classification}, as an exceptional case of a more general goal which we now introduce.

Csajb\'{o}k speculated in~\cite{Csajbok:inverse-closed} that for $\F_q$-subspaces $A$ and $B$ of the same fixed
dimension $d>3$ the stronger bound
$|A^{-1}\cap B|\le |B|/q\cdot\bigl(1+O(q^{-1/2})\bigr)$ might hold.
Our main result shows that this is indeed the case.

\begin{theorem}\label{thm:d_large}
Let $A$ and $B$ be finite nonzero $\F_q$-subspaces of $\barFq$,
with $|A|=|B|=q^d>q^3$ and $A^{-1}\not\subseteq B$.
Then
\[
|A^{-1}\cap B|\le
q^{d-1}+(d-1)(d-2)q^{d-(3/2)}+C_d\cdot q^{d-2},
\]
where $C_d$ only depends on $d$.
\end{theorem}

It follows, in particular, that Csajb\'{o}k's general bound of Theorem~\ref{thm:bound} can only be sharp for subspaces
of dimension up to three (at least for large $q$).

\begin{cor}\label{cor:never}
Equality in the bound $|A^{-1}\cap B|\le 2|B|/q-2$ of Theorem~\ref{thm:bound} is
never attained for $d>3$, where $|B|=q^d$, provided $q$ is sufficiently large with respect to $d$.
\end{cor}

The bound of Theorem~\ref{thm:d_large}, which we prove in Section~\ref{sec:higher-dim},
results from an application of the Lang-Weil bound to
a multivariate polynomial closely related to the polynomial $C(x)$
used in our proof of Theorem~\ref{thm:bound}, when that is irreducible.
However, considerable work is necessary, which we postpone to Section~\ref{sec:form}, to show that such polynomial
can only be reducible in very special situations.
Those exceptional geometric situations generalise the configurations of pairs of two- or three-dimensional subspaces attaining equality in Csajb\'{o}k's bound
which we have briefly discussed earlier.

\section{Intersecting a subspace with the inverse of another}\label{sec:meet}

A finite subset of $\barFq$ is conveniently characterised by the unique monic polynomial
in $\barFq[x]$ whose roots are the elements of the subset, each with multiplicity one.
Thus, to the $\F_q$-subspaces $A$ and $B$ of $\barFq$, with size $q^d$,
throughout the paper we associate the monic polynomials which have the elements of $A$ and $B$,
respectively, as their roots, each with multiplicity one.
(Using the same letters for the subspaces $A$ and $B$ and their polynomials should create no confusion.)
It is well known that $A(x)$ and $B(x)$ are $q$-polynomials,
see~\cite[Theorem~3.52]{LN}, which means that they have the form
$A(x)=\sum_{i=0}^da_ix^{q^i}$
and $B(x)=\sum_{i=0}^db_ix^{q^i}$,
with $a_d=b_d=1$.
Also, the simplicity of their roots amounts to $a_0b_0\neq 0$.
Hence the roots of $x^{q^d}\,A(1/x)=\sum_{i=0}^da_ix^{q^d-q^i}$ and $B(x)/x$ are the
elements of $A^{-1}$ and $B\setminus\{0\}$, respectively.
With this notation at hand we now present a very short proof of Csajb\'{o}k's bounds.

\begin{proof}[Proof of Theorem~\ref{thm:bound}]
The idea of the proof is to give an upper bound on the degree of
the greatest common divisor of $x^{q^d}\,A(1/x)$ and $B(x)/x$.
The fact that the non-leading terms of $B(x)$ have relatively small
degree suggests applying a variant of polynomial long division,
where one keeps subtracting a scalar multiple of $B(x)/x$
from the current reminder multiplied by the appropriate power of $x$.
The final result of this process is condensed in the following argument.

The common roots of
$x^{q^d}\,A(1/x)$ and $B(x)/x$
are also roots of the polynomial
\begin{equation*}
\begin{aligned}
C(x)
&=
x^{q^d}A(1/x)\cdot x^{q^{d-1}-1}-B(x)/x
\cdot
x^{q^{d-1}}\bigl(A(1/x)-1/x^{q^d}\bigr)
\\&=
x^{q^{d-1}-1}
-\bigl(B(x)-x^{q^d}\bigr)/x
\cdot
x^{q^{d-1}}\bigl(A(1/x)-1/x^{q^d}\bigr),
\end{aligned}
\end{equation*}
which has degree at most $2q^{d-1}-2$.
This shows that
$|A^{-1}\cap B|\le 2|B|/q-2$,
except when the polynomial $C(x)$ vanishes.
The latter condition occurs exactly when
$A(x)=x^{q^d}+a_0x$ and $B(x)=x^{q^d}+a_0^{-1}x$,
which means that $A$ and $B$ are $\F_{q^d}$-subspaces, and $B=A^{-1}\cup\{0\}$.

Assuming $d>1$ we now prove a different bound, which holds for arbitrary $q$ but improves on the previous bound only when $q=2$.
The polynomial
\begin{align*}
D(x)
&=
C(x)\cdot x^{q^d-2q^{d-1}+1}+b_{d-1}x^{q^d}\,A(1/x)
\\&=
x^{q^d-q^{d-1}}+b_{d-1}
-\bigl(B(x)-x^{q^d}-b_{d-1}x^{q^{d-1}}\bigr)
\cdot
x^{q^d-q^{d-1}}\bigl(A(1/x)-1/x^{q^d}\bigr),
\end{align*}
has degree at most
$q^d-q^{d-1}+q^{d-2}-1$,
and is nonzero if
$A$ and $B$ are not both $\F_{q^d}$-subspaces.
Because the common roots of
$x^{q^d}\,A(1/x)$ and $B(x)/x$
are also roots of $D(x)$ we obtain
$|A^{-1}\cap B|\le q^d-q^{d-1}+q^{d-2}-1$.
This is better than the previous bound only when $q=2$, and then reads
$|A^{-1}\cap B|\le 3|B|/4-1$.
\end{proof}

The above proof offers more advantages
over Csajb\'{o}k's original proof than just brevity.
It provides us with a polynomial
\begin{equation}\label{eq:C(x)}
C(x)=x^{q^{d-1}-1}-
\bigl(\sum_{i=0}^{d-1}a_ix^{q^{d-1}-q^i}\bigr)\cdot
\bigl(\sum_{j=0}^{d-1}b_jx^{q^j-1}\bigr),
\end{equation}
of degree at most $2q^{d-1}-2$,
such that all elements of $A^{-1}\cap B$ are roots of $C(x)$.
We will put that to good use in the next sections.


We mention in passing that our proof of Theorem~\ref{thm:bound} can be easily modified to deal with affine $d$-dimensional $\F_q$-subspaces of $\barFq$.
In fact, such affine subspaces have the form $A+\alpha$ and $B+\beta$, with $A,B$ as in Theorem~\ref{thm:bound} and $\alpha,\beta\in\barFq$,
which are the sets of roots of the polynomials $A(x-\alpha)=A(x)-A(\alpha)$ and $B(x)-B(\beta)$.
Therefore, the common elements of $(A+\alpha)^{-1}$ and $B+\beta$ are roots of the polynomial
\[
xC(x)
+A(\alpha)x^{q^{d-1}}\bigl(B(x)-x^{q^d}\bigr)
+B(\beta)x^{q^{d-1}}\bigl(A(1/x)-1/x^{q^d}\bigr),
\]
and hence
$|(A+\alpha)^{-1}\cap (B+\beta)|\le 2|B|/q$.
However, we will not consider affine subspaces of $\barFq$ any further in this paper.

We conclude this section by mentioning an alternate, more direct proof of Csajb\'{o}k's bound
in the two-dimensional case.
Let $A$ and $B$ be arbitrary two-dimensional subspaces of $\barFq$,
and consider a maximal set of $\F_q$-linearly independent elements in $A^{-1}\cap B$.
If that set is empty or a singleton, then $|A^{-1}\cap B|$ equals $0$ or $q-1$.
Otherwise, that set consists of $1/\xi$ and $1/\eta$, for some $\xi,\eta\in A$.
If $|A^{-1}\cap B|>2q-2$, then the inverse of some nontrivial linear
combination of $\xi$ and $\eta$ also belongs to $B$.
After possibly scaling $\xi$ or $\eta$ we may assume that $1/(\xi+\eta)$ belongs to $B$,
and hence $1/(\xi+\eta)=a/\xi+b/\eta$ for some nonzero $a, b\in\F_q$.
Therefore, $\xi/\eta$ satisfies a quadratic equation with
coefficients in $\F_q$, and hence $\xi/\eta\in\F_{q^2}$.
(A generalisation of this argument is presented in~\cite{Mat:caps}.)
Consequently, $A$ is an $\F_{q^2}$-subspace,
and $B=A^{-1}$ follows.
We have shown that $|A^{-1}\cap B|/(q-1)\in\{0,1,2,q+1\}$.

In particular, this argument proves an easy fact which is mentioned right after~\cite[Proposition~4.5]{Csajbok:inverse-closed}:
any pair $(A,B)$ of two-dimensional $\F_q$-subspaces of $\barFq$
which attain equality in Csajb\'{o}k's bound $|A^{-1}\cap B|\le 2q-2$,
is obtained by taking as $A$ an arbitrary two-dimensional $\F_q$-subspace
which is not an $\F_{q^2}$-subspace, and as $B$ the $\F_q$-span of the inverses of any two $\F_q$-linearly
independent elements of $A$.

\section{A special configuration of subspaces}\label{sec:three_special}

In this section we consider two $\F_q$-subspaces $A$ and $B$
of $\barFq$, of dimension $d$, in a rather special configuration.
That assumption insures that the polynomial $C(x)$ of Equation~\eqref{eq:C(x)}
has a factor of the form $x^{q^{d-1}-1}+c$,
and this allows precise control over the set of roots of $C(x)$.
This may seem like a rather artificial situation
but, as we will see later in Theorem~\ref{thm:q^3_count}, when $d=3$
it includes all cases where equality is attained in Csajb\'{o}k's bound.
Our proof is purely algebraic, but after the proof we will explain what
goes on in geometric terms.

Because
$(\gamma^{-1}A)^{-1}\cap\gamma B=\gamma(A^{-1}\cap B)$
for any $\gamma\in\barFq^\ast$,
we declare the ordered pair of $\F_q$-subspaces
$(\gamma^{-1}A,\gamma B)$
to be {\em equivalent} to the pair $(A,B)$.
Note that in principle one may consider a weaker equivalence relation
which includes the application of field automorphisms, and possibly interchanging $A$ and $B$, but we chose not to do so.

\begin{theorem}\label{thm:q^3_full}
Let $A$ and $B$ be $\F_q$-subspaces of $\barFq$ of size $q^3$,
and suppose that $(A^{-1}\cap B)\cup\{0\}$ contains a one-dimensional $\F_{q^2}$-subspace of $\barFq$.
Then the pair $(A,B)$ is equivalent to a pair of subspaces consisting of the roots
of the polynomials
$A(x)=x^{q^3}+ax^{q^2}-x^q-ax$
and
$B(x)=x^{q^3}+bx^{q^2}-x^q-bx$, for some $a,b\in\barFq^\ast$, and we have
\begin{enumerate}
\item
$|A^{-1}\cap B|=2q^2-2$
if $q$ is odd, $a^{q+1}=b^{q+1}=-1$, $a^{(q+1)/2}\neq b^{(q+1)/2}$,
and $ab\neq 1$;
\item
$|A^{-1}\cap B|=2q^2-2q$
if $q$ is odd, $a^{q+1}=-1$, and
$a^{(q+1)/2}=b^{(q+1)/2}$;
\item
$|A^{-1}\cap B|=2q^2-q-1$
if $q$ is even,
$a^{q+1}=b^{q+1}=1$,
and $ab\neq 1$;
\item
$|A^{-1}\cap B|\le q^2+2q-3$ otherwise.
\end{enumerate}
\end{theorem}

\begin{proof}
With notation as in the proof of Theorem~\ref{thm:bound}, let
$A(x)=x^{q^3}+a_2x^{q^2}+a_1x^q+a_0x$
and
$B(x)=x^{q^3}+b_2x^{q^2}+b_1x^q+b_0x$
be the monic polynomials with distinct roots, hence with $a_0b_0\neq 0$,
which have $A$ and $B$ as their sets of roots.
That proof shows that all elements of
$A^{-1}\cap B$ are roots of the polynomial
\begin{align*}
C(x)
&=
-a_0b_2x^{2q^2-2}-a_1b_2x^{2q^2-q-1}-a_0b_1x^{q^2+q-2}
\\&\qquad
+(1-a_2b_2-a_1b_1-a_0b_0)x^{q^2-1}
\\&\qquad
-a_1b_0x^{q^2-q}-a_2b_1x^{q-1}-a_2b_0.
\end{align*}

By hypothesis $(A^{-1}\cap B)\cup\{0\}$ contains a one-dimensional
$\F_{q^2}$-subspace, hence the set of roots of a polynomial of the form $x^{q^2}+cx$.
After replacing $A$, $B$ with an equivalent pair we
may assume that $c=-1$, which means assuming that the $\F_{q^2}$-subspace under consideration is the subfield $\F_{q^2}$.
This means that $x^{q^2}-x$ divides both $A(x)$ and $B(x)$, whence
easily
$A(x)=x^{q^3}+ax^{q^2}-x^q-ax$
and
$B(x)=x^{q^3}+bx^{q^2}-x^q-bx$,
and so
\begin{align*}
C(x)
&=
abx^{2q^2-2}+bx^{2q^2-q-1}-ax^{q^2+q-2}
\\&\qquad
-2abx^{q^2-1}
\\&\qquad
-bx^{q^2-q}+ax^{q-1}+ab.
\\&=
(x^{q^2-1}-1)
(abx^{q^2-1}+bx^{q^2-q}-ax^{q-1}-ab).
\end{align*}

Because all polynomials involved can be expressed as polynomials in $x^{q-1}$ we conveniently set
$y=x^{q-1}$, and so
\[
C(x)=(y^{q+1}-1)(aby^{q+1}+by^q-ay-ab).
\]
The derivative criterion shows that the second factor of $C(x)$ shown above has distinct
roots unless $ab=1$, in which case it equals $(y^q-a)(y+a^{-1})$, that is to say,
$(y-a^{1/q})^q(y+a^{-1})$.
In that case $C(x)$ has at most $q^2+2q-3$ distinct roots in $\barFq$
(as a polynomial in $x$),
and hence $|A^{-1}\cap B|\le q^2+2q-3$,
as claimed in assertion (4) of the theorem.
(This can be improved to $|A^{-1}\cap B|\le q^2+q-2$ when $a^{q+1}=\pm 1$,
because then the binomial $y+a^{-1}$ divides one of the other factors $y^{q+1}-1$ and $y^q-a$ of $C(x)$.)

Assume $ab\neq 1$ from now on.
Because
\[
(y^{q+1}+a^{-1}y^q-b^{-1}y-1)\cdot y
-(y^{q+1}-1)\cdot (y+a^{-1})
=-b^{-1}y^2+a^{-1}
\]
we see that the two exhibited factors of $C(x)$ are coprime unless
$(b/a)^{(q+1)/2}=1$ for $q$ odd,
and unless $(b/a)^{q+1}=1$ for $q$ even,
and their greatest common divisor equals
$x^{2q-2}-b/a$ in those cases.
Note that this has distinct roots when $q$ is odd, but it has $q-1$ double roots when $q$ is even.
This will account for the distinction between assertions~(1), (2), and~(3) of the theorem.

Our next task is to find the degree of the greatest common divisor of
$x^{q^3}A(1/x)$ and $C(x)$.
To this goal we compute the remainder of the polynomial $x^{q^3}A(1/x)$ modulo
$C(x)/(x^{q^2-1}-1)$.
All congruences in the remainder of the proof will tacitly be modulo the latter
polynomial, that is, modulo its scalar multiple
$y^{q+1}+a^{-1}y^q-b^{-1}y-1$.
We have
\begin{align*}
-bx^{q^3}A(1/x)
&=
aby^{q^2+q+1}+by^{q^2+q}-aby^{q^2}-b
\\&=
(aby^{q+1}+by^q-ab)y^{q^2}-b
\\&\equiv
ay\cdot y^{q^2}-b.
\end{align*}
Now note that
$y^q\equiv
(b^{-1}y+1)/(y+a^{-1})$,
where our assumption $ab\neq 0$ ensures that the denominator is coprime with the modulus.
Consequently, we have
\[
y^{q^2}
\equiv
\frac{b^{-q}y^q+1}{y^q+a^{-q}}
\equiv
\frac{
 b^{-q}\dfrac{b^{-1}y+1}{y+a^{-1}}
 +1}
{\dfrac{b^{-1}y+1}{y+a^{-1}}
 +a^{-q}}
=
\frac{(1+b^{-q-1})y+(a^{-1}+b^{-q})}
 {(a^{-q}+b^{-1})y+(a^{-q-1}+1)}.
\]
Substituting this into our previous congruence we find
\begin{equation}\label{eq:frac}
\begin{split}
-bx^{q^3}A(1/x)
&\equiv
ay\cdot
\frac{(1+b^{-q-1})y+(a^{-1}+b^{-q})}
 {(a^{-q}+b^{-1})y+(a^{-q-1}+1)}
 -b
\\&=
\frac{a(1+b^{-q-1})y^2
+(ab^{-q}-a^{-q}b)y
-(a^{-q-1}+1)b}
 {(a^{-q}+b^{-1})y+(a^{-q-1}+1)}.
\end{split}
\end{equation}
Hence the greatest common divisor of $x^{q^3-1}A(1/x)$ and
$y^{q+1}+a^{-1}y^q-b^{-1}y-1$
divides the numerator of this expression.

If that numerator vanishes, that is, if $a^{q+1}=b^{q+1}=-1$, then
the factor $y^{q+1}+a^{-1}y^q-b^{-1}y-1$ of $C(x)$
divides $x^{q^3}A(1/x)$, and by assumption so does the other factor $y^{q+1}-1$ of $C(x)$.
We conclude that the greatest common divisor of $x^{q^3}A(1/x)$
and $B(x)$, which divides $C(x)$ but has distinct roots, equals the least common multiple of
the two factors $x^{q^2-1}-1$
and $x^{q^2-1}+a^{-1}x^{q^2-q}-b^{-1}x^{q-1}-1$ of $C(x)$.
According to an earlier calculation, when $q$ is odd this
has degree $2q^2-2$ if $a^{(q+1)/2}\neq b^{(q+1)/2}$,
and $2q^2-2q$ otherwise,
proving assertions~(1) and~(2) of the theorem.
When $q$ is even it has degree $2q^2-q-1$,
as stated in assertion~(3).

If the numerator of the expression found in Equation~\eqref{eq:frac} does not vanish, then
the greatest common divisor of $x^{q^3}A(1/x)$ and $B(x)$
divides the product of that numerator and $x^{q^2-1}-1$, whence
$|A^{-1}\cap B|\le q^2+2q-3$,
as claimed in assertion~(4).
\end{proof}

We briefly pause to comment on the geometric interpretation of the intersection $A^{-1}\cap B$
in the case considered above, as a subset of the three-dimensional
$\F_q$-space $B$.
With $A(x)$ and $B(x)$ as in Theorem~\ref{thm:q^3_full}, we have
\[
x^{q+2}\cdot C(x)
=
(x^{q^2}-x)
(abx^{q^2+q}+bx^{q^2+1}-ax^{2q}-abx^{q+1}).
\]
The monomials $x$, $x^q$ and $x^{q^2}$ determine $\F_q$-linear
maps $B\to\barFq$,
and so they uniquely extend to independent $\barFq$-linear
coordinates $x_0$, $x_1$ and $x_2$ on the linear space $B\otimes_{\F_q}\barFq$.
With this interpretation, the two factors in the above factorisation of $x^{q+2}\cdot C(x)$
may be viewed as representing a linear form and a quadratic form on
$B\otimes_{\F_q}\barFq$, namely, $x_2-x_0$ and
$abx_1x_2+bx_0x_2-ax_1^2-abx_0x_1$.
The quadratic form is nonsingular provided $ab\neq 1$,
as we assume in the rest of this discussion.

In the projective plane $\Proj^2(\barFq)$ associated with the linear space
$B\otimes_{\F_q}\barFq$, this means that the roots of $C(x)$
represent the $\F_q$-rational points of the union of a line, which is defined over $\F_q$, and a
nonsingular conic, which may or may not be defined over $\F_q$.

The first three assertions of Theorem~\ref{thm:q^3_full}
correspond to the case where the conic is defined over $\F_q$, and the further distinction depends on
whether the line is external, secant, or tangent to the conic, with the first two cases occurring only for $q$ odd,
and the last case only for $q$ even.
An alternate presentation of this configuration for $A^{-1}\cap B$, using ideas from finite geometries and limited to $q$ odd, is given in Section~4 of
Csajb\'{o}k's paper~\cite{Csajbok:inverse-closed}.

To complete our geometric interpretation of Theorem~\ref{thm:q^3_full}, when the conic under consideration is not defined over $\F_q$,
its $\F_q$-rational points belong also to the conic obtained from
it by applying the Frobenius map $\alpha\mapsto\alpha^q$ to its coefficients,
and so they are at most four, as they lie on the intersection of
two distinct nonsingular conics.
However, a simple calculation, of which we will sketch a more complex version for cubics in the proof of Theorem~\ref{thm:q^3_Weil},
shows that at most two of the intersection points of the conics over $\barFq$ are $\F_q$-rational.
Adding to those the number of points on the line provides a geometric interpretation for the bound
$|A^{-1}\cap B|/(q-1)\le q+3$ obtained for that case in~Theorem~\ref{thm:q^3_full}.

Part of the argument in the proof of Theorem~\ref{thm:q^3_full}
applies to pairs of higher-dimensional
subspaces in a special configuration described in the following result,
which will be needed later, in the proof of Theorem~\ref{thm:d_large}.

\begin{theorem}\label{thm:q^d}
Let $A$ and $B$ be $\F_q$-subspaces of $\barFq$ of size $q^d\ge q^3$,
and suppose that $(A^{-1}\cap B)\cup\{0\}$ contains a one-dimensional $\F_{q^{d-1}}$-subspace of
$\barFq$.
Then
$|A^{-1}\cap B|\le q^{d-1}+2q^2-3$.
\end{theorem}

\begin{proof}
We argue in a similar way as in the proof of Theorem~\ref{thm:q^3_full}.
After replacing $A$, $B$ with an equivalent pair we
may assume that $x^{q^d}-x$ divides both $A(x)$ and $B(x)$, whence
$A(x)=x^{q^d}+ax^{q^{d-1}}-x^q-ax$
and
$B(x)=x^{q^d}+bx^{q^{d-1}}-x^q-bx$,
with $ab\neq 0$, and so
\begin{align*}
C(x)
&=
(x^{q^{d-1}-1}-1)
(abx^{q^{d-1}-1}+bx^{q^{d-1}-q}-ax^{q-1}-ab).
\end{align*}
We shall compute the remainder of the polynomial $x^{q^4}A(1/x)$
modulo the second factor of $C(x)$.
Working modulo that polynomial we have
\begin{align*}
-bx^{q^d}A(1/x)
&=
abx^{q^d-1}+bx^{q^d-q}-abx^{q^d-q^{d-1}}-b
\\&=
(abx^{q^{d-1}-1}+bx^{q^{d-1}-q}-ab)x^{q^d-q^{d-1}}-b
\\&\equiv
ax^{q-1}\cdot x^{q^d-q^{d-1}}-b.
\end{align*}
Now we conveniently set $y=x^{q-1}$.
Because
$y^{q^{d-2}+\cdots+q}\equiv
(b^{-1}y+1)/(y+a^{-1})$,
we have
\begin{align*}
y^{q^{d-1}+\cdots+q^2}
&\equiv
\frac{b^{-q}y^q+1}{y^q+a^{-q}},
\end{align*}
and hence
\begin{align*}
-bx^{q^d}A(1/x)
&\equiv
ay\cdot y^{q^{d-1}}-b
\\&=
\frac{1}{y^{q^{d-2}+\cdots+q}}
(
ay\cdot y^q\cdot y^{q^{d-1}+\cdots+q^2}
)-b
\\&\equiv
ay^{q+1}\cdot
\frac{y+a^{-1}}{b^{-1}y+1}\cdot
\frac{b^{-q}y^q+1}{y^q+a^{-q}}
-b
\\&\equiv
\frac{ab^{-q}y^{2q+2}+b^{-q}y^{2q+1}+ay^{q+2}-by^q-a^{-q}y-ba^{-q}}
{(b^{-1}y+1)(y^q+a^{-q})}.
\end{align*}
Because the numerator of this expression is nonzero, its degree $(2q+2)(q-1)=2q^2-2$,
in the original indeterminate $x$, is an upper
bound for the degree of the greatest common divisor of
$x^{q^{d-1}-1}+a^{-1}x^{q^{d-1}-q}-b^{-1}x^{q-1}-1$
and $x^{q^d}A(1/x)$.
The desired conclusion follows by taking the other factor $x^{q^{d-1}-1}-1$ of $C(x)$ into account.
\end{proof}

\section{A better bound for subspaces of dimension at least four}\label{sec:higher-dim}

The proof of Theorem~\ref{thm:bound} which we gave in
Section~\ref{sec:meet} shows that all elements of $A^{-1}\cap B$
are roots of the polynomial $C(x)$
of Equation~\eqref{eq:C(x)}, and hence of the modified polynomial
\begin{align*}
C_0(x)
&=x^{1+(1+q+q^2+\cdots+q^{d-2})}\cdot C(x)
\\&=
x^{1+q+q^2+\cdots+q^{d-1}}
\biggl(1-
\bigl(\sum_{i=0}^{d-1}a_i/x^{q^i}\bigr)\cdot
\bigl(\sum_{j=0}^{d-1}b_jx^{q^j}\bigr)
\biggr).
\end{align*}
The latter has the advantage of being a linear combination of monomials,
each of whose degrees is a sum of $d$ terms taken from the set
$\{1,q,\ldots,q^{d-1}\}$ with at most one repetition,
and hence one omission.
We now show how being a root of $C_0(x)$ can be interpreted as being a zero of one or more homogeneous forms of degree $d$
on the $\F_q$-space $B$.

The $\F_q$-linear maps $x_i:B\to\barFq$ given by $x\mapsto x^{q^i}$, for $0\le i<d$,
are $\F_q$-linearly independent, and so they
form a complete set of linear coordinates (that is, a basis of the dual space) on the $\barFq$-space
$B\otimes_{\F_q}\barFq$.
Hence to $C_0(x)$ there corresponds a homogeneous polynomial
function
\[
E(x_0,\ldots,x_{d-1})
=
x_0\cdots x_{d-1}\cdot
\biggl(1-
\bigl(\sum_{i=0}^{d-1}a_i/x_i\bigr)\cdot
\bigl(\sum_{j=0}^{d-1}b_jx_j\bigr)
\biggr),
\]
of degree $d$, defined on the $\barFq$-space
$B\otimes_{\F_q}\barFq$.
We would rather need a polynomial function on the original $\F_q$-space $B$, but we can obtain that
through a linear change of coordinates.
In fact, an arbitrary $\F_q$-linear map on $B$ with values in $\F_q$
is given by
\begin{multline*}
x
\mapsto
b_0\gamma x
+(b_0^q\gamma^q+b_1\gamma)x^q
+(b_0^{q^2}\gamma^{q^2}+b_1^q\gamma^q+b_2\gamma)x^{q^2}
+\cdots
\\
\cdots
+(b_0^{q^{d-1}}\gamma^{q^{d-1}}+b_1^{q^{d-2}}\gamma^{q^{d-2}}+\cdots+b_{d-1}\gamma)x^{q^{d-1}},
\end{multline*}
where $\gamma$ ranges over the roots of the $q$-polynomial
$b_0^{q^d}x^{q^d}+b_1^{q^{d-1}}x^{q^{d-1}}+\cdots+b_dx$.
Therefore, a complete set of
$\F_q$-linear coordinates $z_0,\ldots,z_{d-1}$ on the $\F_q$-space $B$
is obtained by letting $\gamma$ range over an $\F_q$-basis of the
roots of that $q$-polynomial.
After expressing each $x_i$ in terms of $z_0,\ldots,z_{d-1}$ (as linear combinations over $\barFq$),
the polynomial function $E(x_0,\ldots,x_{d-1})$ on $B\otimes_{\F_q}\barFq$ yields a polynomial function
$\tilde E(z_0,\ldots,z_{d-1})$ on $B$, but still with values in $\barFq$,
which is homogeneous of degree $d$ in $z_0,\ldots,z_{d-1}$.

To recapitulate in slightly different wording, all elements of $A^{-1}\cap B$,
once $B$ is identified with $\F_q^d$ via the coordinates $z_i$,
are roots of the polynomial
$\tilde E(z_0,\ldots,z_{d-1})\in\barF_q[z_0,\ldots,z_{d-1}]$.
According to Theorem~\ref{thm:form} below, this polynomial turns out to be usually irreducible for our purposes,
with notable exceptions where our geometric problem has already
been dealt with in Section~\ref{sec:three_special}.
When the polynomial is indeed irreducible, it defines a hypersurface in the projective space
$\Proj(B\otimes_{\F_q}\barFq)\cong\Proj^{d-1}(\barFq)$,
whose number of $\F_q$-rational points can be bounded using the
Lang-Weil bound~\cite{Lang-Weil}.
These are the ideas at play in the proof of our
Theorem~\ref{thm:d_large} stated in the Introduction, which we give below,
but not before stating the result on the possible factorisations of $E$ which we have just mentioned.

\begin{theorem}\label{thm:form}
The homogeneous polynomial
\[
E=E(x_1,\ldots,x_n)
=
x_1\cdots x_n\cdot
\biggl(1+
\bigl(\sum_{i=1}^na_i/x_i\bigr)\cdot
\bigl(\sum_{j=1}^nb_jx_j\bigr)
\biggr)
\]
where $a_i,b_j\in\barFq$,
has at most two non-monomial (absolutely) irreducible factors.
If it has two then at least one of them is a linear combination of exactly two of the indeterminates.

Furthermore, if $E$ has $x_1+x_2$ as a factor, then either
\begin{multline*}
E/(x_3\cdots x_n)=
x_1x_2\cdot\bigl(
1+(a/x_1+(a+c)/x_2)\cdot(bx_1+(b-c^{-1})x_2)
\bigr)
\\
=
(x_1+x_2)\bigl((a+c)bx_1+a(b-c^{-1})x_2\bigr)
\end{multline*}
for some $a,b,c\in\barFq$ with $c\neq 0$, or
\begin{multline*}
E/(x_4\cdots x_n)=
x_1x_2x_3\cdot\bigl(
1+(a/x_1+a/x_2-1/x_3)\cdot(bx_1+bx_2-x_3)
\bigr)
\\
=
(x_1+x_2)(abx_2x_3+abx_1x_3+bx_1x_2-ax_3^2),
\end{multline*}
up to permuting the indeterminates $x_3,\ldots,x_n$.
\end{theorem}

The harmless sign change in the definition of $E$ from the previous notation will avoid the occurrence of several minus signs in the proof
of Theorem~\ref{thm:form}.
We have also conveniently shifted the indices of the indeterminates, which now start from one.
The proof of Theorem~\ref{thm:form} is a little technical, and we postpone it to Section~\ref{sec:form}
to avoid disrupting the flow of the present argument.
Note that the second nontrivial factorisation allowed by Theorem~\ref{thm:form}
has already occurred in disguise in the factorisations of $C(x)$
obtained in the proofs of Theorems~\ref{thm:q^3_full} and~\ref{thm:q^d}.

\begin{proof}[Proof of Theorem~\ref{thm:d_large}]
Continue with the setting introduced above.
As we noted in the proof of Theorem~\ref{thm:bound},
the condition $A^{-1}\not\subseteq B$ implies that $C(x)$
is not the zero polynomial, whence
$E(x_0,\ldots,x_{d-1})$ is not the zero polynomial.
Monomial factors of $E(x_0,\ldots,x_{d-1})$
and, correspondingly, of $C(x)$,
clearly give no contribution to estimating $|A^{-1}\cap B|$.

Suppose first that $E(x_0,\ldots,x_{d-1})$ has a unique non-monomial irreducible factor
$F(x_0,\ldots,x_{d-1})$, of degree $d'\le d$.
Let $\tilde F(z_0,\ldots,z_{d-1})$ be the corresponding polynomial
written in terms of the coordinates $z_0,\ldots,z_{d-1}$ on the $\F_q$-space $B$.
It defines an irreducible algebraic subvariety of the
projective space
$\Proj(B\otimes_{\F_q}\barFq)\cong\Proj^{d-1}(\barFq)$,
of dimension $d-2$ (that is, a hypersurface).

If that variety is defined over $\F_q$, which occurs if $\tilde F(z_0,\ldots,z_{d-1})$,
after multiplication by a suitable scalar, can be made to have all coefficients in
$\F_q$, then according to the Lang-Weil estimate~\cite{Lang-Weil}
its number of $\F_q$-rational points is bounded above by
\begin{equation}\label{eq:Lang-Weil}
q^{d-2}+(d'-1)(d'-2)q^{d-(5/2)}+C_{d,d'}\cdot q^{d-3},
\end{equation}
where $C_{d,d'}$ is a constant which depends only on $d$ and $d'$.
The desired conclusion is then obtained upon multiplication by
$q-1$ and using the fact that $d'\le d$.

Now suppose that the variety under consideration is not defined
over $\F_q$.
Then the Galois-conjugate polynomial $\tilde F^\sigma(z_0,\ldots,z_{d-1})$,
obtained from $\tilde F(z_0,\ldots,z_{d-1})$ by applying the Frobenius
automorphism $\sigma:a\mapsto a^q$ to each coefficient,
is not proportional to $\tilde F(z_0,\ldots,z_{d-1})$,
and hence together with the latter it defines a (possibly reducible) algebraic set in
$\Proj^{d-1}(\barFq)$, of dimension strictly smaller than $d-2$,
and degree at most $(d')^2$.
According to a standard fact known as the
{\em Schwartz-Zippel lemma,} see~\cite[Lemma A.3]{Tao+:Kakeya}
for a proof, the number of $\F_q$-rational points of this
algebraic set is at most $(d')^2(q+1)^{d-3}$.
After multiplying by $q-1$ we see that the desired conclusion holds in this case as well.

Now we may assume that $E(x_0,\ldots,x_{d-1})$ has at least two non-monomial irreducible factors.
Then it has exactly two according to Theorem~\ref{thm:form},
and one of them is a linear combination of
two of the indeterminates, say $x_i$ and $x_j$ with $i<j$.
The corresponding factor of our original polynomial $C(x)$
is then a linear combination of $x^{q^i-1}$ and $x^{q^j-1}$,
and hence it accounts for at most $q^{j-i}-1$ distinct nonzero roots of that polynomial.
The possible factorisations of $E$ given in Theorem~\ref{thm:form}
show that the remaining factor of $C(x)$ has degree at most $q^{d-1}-1$,
whence $|A^{-1}\cap B|\le q^{d-1}+q^{j-i}-2$.
If $j-i<d-1$ we have reached our goal, hence assume $j-i=d-1$.
This means that the former factor of
$C(x)$ considered above is a linear combination of $1$ and $x^{q^{d-1}-1}$.
Possibly after replacing $(A,B)$ with an equivalent pair we may assume that linear combination to be $x^{q^{d-1}-1}-1$.

Now consider, in turn, the two possible factorisations of $E$ stated in Theorem~\ref{thm:form},
and what they entail for the polynomials $A(x)$ and $B(x)$ in our setting.
The former factorisation implies
$A(x)=x^{q^d}+ax^{q^{d-1}}-(a+c)x$
and
$B(x)=x^{q^d}-bx^{q^{d-1}}+(b-c^{-1})x$,
whence
\begin{align*}
C(x)
&=
(x^{q^{d-1}-1}-1)
\bigl(-(a+c)bx^{q^{d-1}-1}+a(b-c^{-1})\bigr).
\end{align*}
Because $B(x)\equiv x^q-c^{-1}x\pmod{x^{q^{d-1}-1}-1}$, the polynomial $B(x)$ has at most $q-1$ nonzero roots in common with
the former factor of $C(x)$, and similarly with the latter factor.
Hence in this case we have $|A^{-1}\cap B|\le 2q-2$.

The other possible factorisation of $E$ yields
$A(x)=x^{q^d}+ax^{q^{d-1}}-cx^{q^e}-ax$
and
$B(x)=x^{q^d}+bx^{q^{d-1}}-c^{-1}x^{q^e}-bx$,
with $0<e<d-1$ and $abc\neq 0$, and so
\begin{align*}
C(x)
&=
(x^{q^{d-1}-1}-1)
(abx^{q^{d-1}-1}+bcx^{q^{d-1}-q^e}-ac^{-1}x^{q^e-1}-ab).
\end{align*}
Because $B(x)\equiv x^q-c^{-1}x^{q^e}\pmod{x^{q^{d-1}-1}-1}$, the polynomial $B(x)$ has at most $q-1$ nonzero roots in common with
the former factor of $C(x)$, whence
$|A^{-1}\cap B|\le q^{d-1}+q-2$,
except when $c=1$ and $e=1$.
However, in the latter case Theorem~\ref{thm:q^d} applies and yields
$|A^{-1}\cap B|\le q^{d-1}+2q^2-3$.
\end{proof}

\section{A classification of pairs of three-dimensional subspaces\\
with large intersection $A^{-1}\cap B$}\label{sec:classification}

In the case of subspaces of dimension $d=3$, which was excluded from Theorem~\ref{thm:d_large}, parts of its proof still apply.
In particular, the Lang-Weil estimate of Equation~\eqref{eq:Lang-Weil} takes
the more precise form of the Hasse-Weil bound, and allows us to prove the following result.

\begin{theorem}\label{thm:q^3_Weil}
Let $A$ and $B$ be $\F_q$-subspaces of $\barFq$, of size $q^3$, with $A^{-1}\not\subseteq B$.
If $|A^{-1}\cap B|/(q-1)>q+1+\lfloor 2\sqrt{q}\rfloor$, then
$|A^{-1}\cap B|/(q-1)$ equals either $2q+2$ or $2q$ for $q$ odd, and it equals $2q+1$ for $q$ even.
\end{theorem}

Note that equality in Csajb\'{o}k's bound $|A^{-1}\cap B|\le 2q^2-2$ for three-dimensional spaces cannot be attained in characteristic two.

\begin{proof}
Arguing as in the proof of Theorem~\ref{thm:d_large}, which we gave in Section~\ref{sec:higher-dim}, and aiming at a contradiction,
suppose that $E(x_0,x_1,x_2)$ is irreducible.
Writing this in terms of the coordinates $z_0,\ldots,z_{d-1}$ on the $\F_q$-space $B$
we get a polynomial $\tilde E(z_0,z_1,z_2)$
which defines an irreducible cubic in the projective plane
$\Proj(B\otimes_{\F_q}\barFq)\cong\Proj^{2}(\barFq)$.
If the cubic is defined over $\F_q$, then according to the Hasse-Weil bound its number of $\F_q$-rational points does not exceed
$q+1+2\sqrt{q}$, whence
$|A^{-1}\cap B|/(q-1)\le q+1+2\sqrt{q}$,
which contradicts our hypothesis.
If the cubic is not defined over $\F_q$, then
its intersection with the irreducible cubic defined by
$\tilde E^\sigma(z_0,z_1,z_2)$,
where $\sigma$ is the Frobenius automorphism,
has at most $3^2=9$ points in $\Proj(B\otimes_{\F_q}\barFq)$ according to B\'{e}zout's theorem
(or to the Schwartz-Zippel lemma if we prefer).
Hence $|A^{-1}\cap B|/(q-1)\le 9$,
and we obtain a contradiction because this number does not exceed Weil's bound $q+1+\lfloor 2\sqrt{q}\rfloor$,
except when $q\le 3$.
However, when $q=2$ Theorem~\ref{thm:bound} provides the improved bound
$|A^{-1}\cap B|\le 3\cdot 2^3/4-1=5$, which yields the desired contradiction.

In order to cover the case $q=3$ as well, we sketch how an explicit calculation allows us to
strengthen the upper bound of $9$ given by B\'{e}zout's theorem to the bound
$|A^{-1}\cap B|/(q-1)\le 6$, for arbitrary $q$.
We do that by showing that the intersection of the zero sets of
$\tilde E(z_0,z_1,z_2)$ and its Galois-conjugate
$\tilde E^\sigma(z_0,z_1,z_2)$
contains at least three non-rational points.
One way of computing the Galois-conjugate in terms of the original coordinates $x_0,x_1,x_2$ is
raising the polynomial
\begin{align*}
C_0(x)
=x^{q+2}\cdot C(x)
&=
-a_0b_2x^{2q^2+q}-a_1b_2x^{2q^2+1}-a_0b_1x^{q^2+2q}
\\&\qquad
+(1-a_2b_2-a_1b_1-a_0b_0)x^{q^2+q+1}
\\&\qquad
-a_1b_0x^{q^2+2}-a_2b_1x^{2q+1}-a_2b_0x^{q+2}
\end{align*}
to the $q$-th power and reducing the result modulo the polynomial $B(x)$.
Writing both the remainder of this division and the original polynomial $C_0(x)$ in terms of $x_0=x$, $x_1=x^q$, and $x_2=x^{q^2}$,
one discovers that both vanish for
$(x_0,x_1,x_2)=(1,0,0),(b_1,-b_0,0),(b_2,0,-b_0)$.
However, a triple $(x_0,x_1,x_2)$ gives a rational point of our curve, that is, an element of $B$,
only when $x_1=x_0^q$, $x_2=x_1^q$, and $x_2^q=-b_2x_2-b_1x_1-b_0x_0$,
which is not the case for any of the three triples found.

Thus, we have shown that $E(x_0,x_1,x_2)$ cannot be irreducible.
Invoking Theorem~\ref{thm:form} and arguing as in the proof of Theorem~\ref{thm:d_large}
we see that $|A^{-1}\cap B|$ can possibly be so large only when $C(x)$
has a non-trivial linear combination of $1$ and $x^{q^2}$ as a factor,
which may be taken to be $x^{q^2-1}-1$ after passing to an equivalent pair $(A,B)$.
The proof of Theorem~\ref{thm:d_large} also shows that we must have
the second exceptional factorisation of $E(x_0,x_1,x_2)$ given in Theorem~\ref{thm:form},
and that
$A(x)=x^{q^3}+ax^{q^2}-x^q-ax$
and
$B(x)=x^{q^3}+bx^{q^2}-x^q-bx$.
Consequently, both $A$ and $B$ contain $\F_{q^2}$, hence Theorem~\ref{thm:q^3_full} applies
and completes the proof.
\end{proof}

\begin{rem}\label{rem:comparison}
A special version of Theorem~\ref{thm:q^3_Weil}
occurs as assertion~(1) of~\cite[Theorem~4.8]{Csajbok:inverse-closed}.
That result restricts $A$ and $B$ to be contained in $\F_{q^4}$,
which is a rather strong assumption.
Note that~\cite[Theorem~4.8]{Csajbok:inverse-closed}
is actually proved under the unstated hypothesis that $q$ is odd,
but fails to exclude the possibility that $|A^{-1}\cap B|/(q-1)=2q+1$
(except in the special case where $A=B$).
At the author's request Csajb\'{o}k has produced a proof which excludes that possibility, based on similar methods as~\cite{Csajbok:inverse-closed}.
\end{rem}

Theorem~\ref{thm:q^3_Weil}, combined with the special configuration which we investigated in Theorem~\ref{thm:q^3_full},
allows us to classify all pairs $(A,B)$ attaining equality in Csajb\'{o}k's bound, or almost, in the three-dimensional case.

\begin{theorem}\label{thm:q^3_count}
In the following assertions $A$ and $B$ denote $\F_q$-subspaces of $\barFq$, of size $q^3$, with $A^{-1}\not\subseteq B$.
\begin{enumerate}
\item
There are exactly $(q-1)/2$ equivalence classes of pairs $(A,B)$
such that $|A^{-1}\cap B|=2q^2-2$.
\item
For odd $q>5$ there are exactly $(q+1)/2$ equivalence classes of pairs $(A,B)$
such that $|A^{-1}\cap B|=2q^2-2q$.
\item
For even $q>4$ there are exactly $q$ equivalence classes of pairs $(A,B)$
such that $|A^{-1}\cap B|=2q^2-q-1$.
\item
Each equivalence class described in assertions~(1) and~(2) contains exactly two pairs satisfying $A=B$,
and each equivalence class described in assertion~(3) contains exactly one such pair.
Each such subspace $A=B$ is contained in $\F_{q^4}$, and equals the kernel of
$x\mapsto\Tr_{\F_{q^4}/\F_q}(\alpha x)$,
for some $\alpha\in\F_{q^4}$ with $\alpha^2\in\F_{q^2}\setminus\F_q$.
\end{enumerate}
\end{theorem}

\begin{proof}The stated conditions on $q$ insure that $|A^{-1}\cap B|/(q-1)$ exceeds the Hasse-Weil bound in each case.
Therefore, as in the proof of Theorem~\ref{thm:q^3_Weil} we conclude that after replacing $(A,B)$
with an equivalent pair we have
$A(x)=x^{q^3}+ax^{q^2}-x^q-ax$
and
$B(x)=x^{q^3}+bx^{q^2}-x^q-bx$,
and Theorem~\ref{thm:q^3_full} gives exact conditions that $a$ and $b$ satisfy.

Now the subspaces $\gamma^{-1}A$ and $\gamma B$, which form an equivalent pair to $(A,B)$ for $\gamma\in\barF_q^\ast$,
are the sets of roots of the monic polynomials $A_{\gamma^{-1}}(x)=\gamma^{-q^3}A(\gamma x)$ and $B_{\gamma}(x)=\gamma^{q^3}B(x/\gamma)$.
Comparing coefficients we see that each of the equalities
$\gamma^{-1}A=A$ and $\gamma B=B$
occurs only when $\gamma^{q-1}=1$, that is, when $\gamma\in\F_q^\ast$.
However, $A_{\gamma^{-1}}(x)$ has the admissible form $A_{\gamma^{-1}}(x)=x^{q^3}+a'x^{q^2}-x^q-a'x$ considered above
(which means that it is a multiple of $x^{q^2}-1$) if and only if $\gamma\in\F_{q^2}^{\ast}$,
and so does $B_{\gamma}(x)$.
Consequently, each equivalence class of pairs $(A,B)$ under consideration contains exactly
$q+1$ pairs for which the corresponding polynomials $A(x)$ and $B(x)$ have the required form.
Thus, the number of equivalence classes is obtained after dividing by $q+1$ the number of pairs
$(a,b)$ with $a^{q+1}=b^{q+1}=-1$ and $ab\neq 1$, and possibly the further conditions
given in Theorem~\ref{thm:q^3_full}.
This establishes assertions~(1), (2), and~(3).

A similar coefficient comparison shows that for the pairs $(A,B)$ under consideration
an equivalent pair $(\gamma^{-1}A,\gamma B)$ satisfies $\gamma^{-1}A=\gamma B$ exactly when
$\gamma^{2(q-1)}=a/b$.
Consequently, there are two such pairs equivalent to $(A,B)$ when $q$ is odd,
and only one when $q$ is even, as claimed in assertion~(4).
Furthermore, taking $\gamma^{2(q-1)}=a/b$, whence $\gamma^{q^2-1}=(a/b)^{(q+1)/2}=\pm 1$, we obtain
\begin{align*}
A_{\gamma^{-1}}(x)
=B_{\gamma}(x)
&=x^{q^3}+\gamma^{q^3-q^2}bx^{q^2}-\gamma^{q^3-q}x^q+\gamma^{q^3-1}bx
\\&=x^{q^3}+\gamma^{q-1}bx^{q^2}-\gamma^{q^2-1}(x^q+\gamma^{q-1}bx)
\\&=x^{q^3}+cx^{q^2}+c^{q+1}(x^q+cx)
\\&=x^{q^3}+cx^{q^2}+c^{q+1}x^q+c^{q^2+q+1}x,
\end{align*}
having set $c=\gamma^{q-1}b$ and noted that $c^{q+1}=-\gamma^{q^2-1}=\pm 1$.
For $(a,b)$ ranging over all pairs such that
$a^{q+1}=b^{q+1}=-1$ and $ab\neq 1$, and with $\gamma$ chosen as above,
$c$ takes all the values such that $c^{2(q+1)}=1$ and $c^2\neq 1$.
The description of $A$ given in assertion~(4) follows at once by writing $c=\alpha^{-(q-1)}$.
\end{proof}

\begin{rem}
The restrictions on $q$ in assertions~(2) and~(3) of Theorem~\ref{thm:q^3_count} cannot be relaxed.
For example, a computer calculation shows that $\F_{5^4}$ contains $31\cdot 8$ pairs $(A,B)$ of three-dimensional $\F_5$-subspaces
such that $|A^{-1}\cap B|=2q^2-2q=40$, rather than $(q^2+q+1)\cdot(q+1)/2=31\cdot 3$
as Theorem~\ref{thm:q^3_count} would predict.
For those in excess, $A^{-1}\cap B$ yields an irreducible cubic in the projective plane $\Proj B$.
\end{rem}

\begin{rem}\label{rem:caps}
In case of assertion~(1) of Theorem~\ref{thm:q^3_count}, where equality in Csajb\'{o}k's bound is attained,
an alternate approach is available and described in~\cite{Mat:caps}.
It relies on facts from finite geometries to bypass the arguments of this and the previous section,
and hence Theorem~\ref{thm:form}, on which they ultimately depend.
Briefly, as a special case of a more general result it is shown in~\cite{Mat:caps} that the image of
$A^{-1}\cap B$ in $\Proj B$ is an arc,
for three-dimensional $\F_q$-subspaces $A,B$ of $\barFq$,
unless $(A^{-1}\cap B)\cup\{0\}$
contains a one-dimensional $\F_{q^2}$-subspace of $\barFq$.
Because it is known that an arc in $\Proj^2(\F_q)$ has at most $2q+1$ points if $q>3$,
when equality $|A^{-1}\cap B|/(q-1)=2q+2$ holds in Csajb\'{o}k's bound
we conclude that $(A^{-1}\cap B)\cup\{0\}$
contains a one-dimensional $\F_{q^2}$-space, and then our Theorem~\ref{thm:q^3_full} applies.
If $q$ is odd (and larger than three), at this point one can also
deduce that the image of $A^{-1}\cap B$ in $\Proj B$ is the union of a line and a conic
without using Theorem~\ref{thm:q^3_full}, appealing instead
to the classical result of B.~Segre that an an arc with $q+1$ points in a projective plane
is a conic (for $q$ odd), see~\cite[Theorem~5]{Mat:caps}.
\end{rem}

\begin{rem}\label{rem:q^2}
In contrast with the three-dimensional case considered in Theorem~\ref{thm:q^3_count}, for a fixed prime power $q$
there are infinitely many equivalence classes of pairs $(A,B)$ of two-dimensional $\F_q$-subspaces
which attain equality in Csajb\'{o}k's bound $|A^{-1}\cap B|\le 2q-2$.
This follows at once from their description which we gave at the end of Section~\ref{sec:meet}.
\end{rem}

\begin{cor}\label{cor:q^3}
Assume $q>5$, and set $P(x)=x^{q^3}+x^{q^2}+x^q+x$.
If $A$ and $B$ are $\F_q$-subspaces of $\barFq$, of size $q^3$, such that
\[
q^2-1+\lfloor 2q^{1/2}\rfloor\cdot(q-1)
<|A^{-1}\cap B|<q^3-1,
\]
then $A$ and $B$ are the sets of roots of $P(\alpha\gamma^{-1}x)$ and $P(\alpha\gamma x)$,
respectively, for some $\alpha,\gamma\in\barFq^{\ast}$
with $\alpha^2\in\F_{q^2}\setminus\F_q$.
\end{cor}

\begin{proof}
According to Theorem~\ref{thm:q^3_Weil} the pair $(A,B)$ is one of those described in Theorem~\ref{thm:q^3_count}.
According to assertion~(4) of the latter, some equivalent pair $(\gamma^{-1}A,\gamma B)$ satisfies
$\gamma^{-1}A=\gamma B$, and that subspace equals the set of roots of $P(\alpha x)$,
for some $\alpha\in\barFq^{\ast}$ with $\alpha^2\in\F_{q^2}\setminus\F_q$.
\end{proof}

The explicit description of the spaces $A$ and $B$ given in Corollary~\ref{cor:q^3} allows one
to decide whether and how many of them can be found inside a given finite field.
For example, in $\F_{q^4}$ there are exactly $2q(q^2+q+1)$ such pairs $(A,B)$ for $q$ odd,
and $q(q^2+q+1)$ for $q$ even, because $\gamma\in\F_{q^4}^{\ast}$ in this case.
Those among them with $A=B$ are in number of $2q$ and $q$, respectively,
and for $q$ odd they match those described in~\cite[Propositions~4.4 and~4.5]{Csajbok:inverse-closed}.

Another example is the following improvement of Csajb\'{o}k's bound for three-dimensional subspaces of finite fields which do not contain $\F_{q^4}$.

\begin{cor}
Consider the finite field $\F_{q^e}$, where $q>5$ and $e$ is not a multiple of four.
If $A$ and $B$ are $\F_q$-subspaces of $\F_{q^e}$ with size $q^3$, and $A^{-1}\not\subseteq B$, then
$
|A^{-1}\cap B|\le q^2-1+\lfloor 2q^{1/2}\rfloor\cdot(q-1).
$
\end{cor}

\begin{proof}
Consider subspaces $A$ and $B$ of $\barFq$ as in Corollary~\ref{cor:q^3}.
When $\gamma=1$ they are both contained in $\F_{q^4}$,
and  $(A^{-1}\cap B)\cup\{0\}$ properly contains the one-dimensional
$\F_{q^2}$-subspace consisting of the roots of $(\alpha x)^{q^2}+\alpha x$.
Therefore, the subfield generated by all the quotients of
pairs of elements of $A^{-1}\cap B$ properly contains
$\F_{q^2}$, and hence equals $\F_{q^4}$.
This last statement carries over to the case of arbitrary $\gamma$.
Consequently, the subfield of $\barFq$ generated by $A^{-1}\cap B$ contains $\F_{q^4}$,
and hence $A$ and $B$ cannot be both contained in $\F_{q^e}$.
\end{proof}

\section{Proof of Theorem~\ref{thm:form}}\label{sec:form}

Recall that a polynomial is called {\em multilinear} if it has degree at most one in each indeterminate.
The following property of the polynomial $E$, which follows from the definition
and is inherited by factors, will be crucial in most of our arguments.

\begin{property}\label{property}
Any factor of $E$ has degree at most two in
each indeterminate $x_i$, and joint degree at most three in each
pair of indeterminates $x_i$ and $x_j$.
\end{property}

Thus, for example, $E$ cannot have any term divisible by $x_1^2x_2^2$.
As another example, if $E=FG$ and $F$ has degree two in some $x_i$,
then $G$ cannot involve $x_i$.

\subsection{Linear factors of $E$}
We first prove the conclusions of Theorem~\ref{thm:form} under the additional assumption
that $E$ has a non-monomial linear factor.

\begin{lemma}\label{lemma:linear}
Under the hypotheses of Theorem~\ref{thm:form},
any non-monomial linear factor of $E$ is a linear combination of exactly two
indeterminates.
\end{lemma}

\begin{proof}
Let $E=FG$, with $G$ a non-monomial linear factor.
Because $F$ is homogeneous of degree $n-1$ in $n$ indeterminates,
according to Property~\ref{property} each term of $F$ misses at most two indeterminates.

Now suppose for a contradiction that $G$ involves at least three indeterminates.
Then an arbitrary monomial of $F$ must share at least one indeterminate with $G$, say $x_i$.
But then $F$ has degree exactly one in that $x_i$, and joint degree at most two in
$x_i$ and $x_j$, for each other indeterminate $x_j$.
Consequently, that arbitrary term of $F$ has degree at most one in each
indeterminate, and hence $F$ is multilinear.

Writing
$F=x_1\cdots x_n\cdot\sum_{i=1}^n f_i/x_i$
and $G=\sum_{j=1}^n g_jx_j$
and comparing coefficients of the non-multilinear terms of $E$ on both sides of the equation $FG=E$
we find
$f_ig_j=a_ib_j$ for $i\neq j$.
Now our assumption that at least three of the coefficients $g_j$ are nonzero
implies that the $n$-tuples $(f_1,\ldots,f_n)$ and $(a_1,\ldots,a_n)$ are proportional.
In fact, our equations imply $(f_ia_j-f_ja_i)g_kb_k=0$ for any distinct $i,j,k$.
If $g_k\neq 0$ for some $k$, then $b_k\neq 0$, otherwise
$f_i=a_ib_k/g_k=0$ and $f_j=a_jb_k/g_k=0$, which is impossible.
Hence if $g_k\neq 0$ for some $k$, then $f_ia_j=f_ja_i$ for any $i,j\neq k$.
Consequently, if $g_k\neq 0$ for at least three values of $k$, then
then $f_ia_j=f_ja_i$ for any $i,j$, and hence
the $n$-tuples $(f_1,\ldots,f_n)$ and $(a_1,\ldots,a_n)$ are proportional.

We conclude that $F$ is a scalar multiple of
$x_1\cdots x_n\cdot \sum_{i=1}^{n}a_i/x_i$,
in plain contradiction with the definition of $E$ and the fact that $FG=E$.
\end{proof}

\begin{lemma}\label{lemma:exceptions}
Under the hypotheses of Theorem~\ref{thm:form}, if $E$
has $x_1+x_2$ as a factor, then either
\begin{multline*}
E/(x_3\cdots x_n)=
x_1x_2\cdot\bigl(
1+(a/x_1+(a+c)/x_2)\cdot(bx_1+(b-c^{-1})x_2)
\bigr)
\\
=
(x_1+x_2)\bigl((a+c)bx_1+a(b-c^{-1})x_2\bigr)
\end{multline*}
for some $a,b,c\in\barFq$ with $c\neq 0$, or
\begin{multline*}
E/(x_4\cdots x_n)=
x_1x_2x_3\cdot\bigl(
1+(a/x_1+a/x_2-1/x_3)\cdot(bx_1+bx_2-x_3)
\bigr)
\\
=
(x_1+x_2)(abx_2x_3+abx_1x_3+bx_1x_2-ax_3^2),
\end{multline*}
up to permuting the indeterminates $x_3,\ldots,x_n$.
\end{lemma}

\begin{proof}
Any non-multilinear factor of $F$ cannot involve either $x_1$ or $x_2$, and hence
\[
F=x_1\cdots x_n\cdot\sum_{i=1}^n f_i/x_i+x_3\cdots x_n\cdot\sum_{j=3}^n f'_jx_j.
\]
Then the product $E=FG$ has no term of the form $x_1\cdots x_n\cdot x_j/x_i$,
and so we have $a_ib_j=0$ whenever $i,j>2$ and $i\neq j$.
This implies $a_i=b_i=0$ for all indices $i>2$ except possibly one,
and possibly after permuting the indeterminates $x_3,\ldots,x_n$ we may assume that
$a_i=b_i=0$ for $i>3$.
Therefore, each of the indeterminates $x_4,\ldots,x_n$ appears in each monomial of $E$
with exponent exactly one, and hence $f_i=f'_i=0$ for $i>3$.

Thus we have $F=(f_1x_2x_3+f_2x_1x_3+f_3x_1x_2+f'_3x_3^2)\cdot x_4\cdots x_n$, and
comparing coefficients of each term on both sides of the equation  $FG=E$ we find
\begin{gather*}
f_2=a_2b_1,
\quad
f_3=a_3b_1,
\quad
f_1=a_1b_2,
\quad
f_3=a_3b_2,
\\
f_1+f_2=1+a_1b_1+a_2b_2+a_3b_3,
\quad
f'_3=a_2b_3,
\quad
f'_3=a_1b_3.
\end{gather*}
Substituting the first and third equation of the set into the fifth one turns that into
$(a_1-a_2)(b_1-b_2)+a_3b_3+1=0$.
Hence if $a_3b_3=0$ then $a_1\neq a_2$ and $b_1\neq b_2$, whence $f_3=f'_3=0$,
and so
\[
F=\bigl(
(a+c)bx_1+a(b-c^{-1})x_2
\bigr)\cdot x_3\cdots x_n,
\]
where we have set $a:=a_1$, $b:=b_1$, and $c:=a_2-a_1=(b_1-b_2)^{-1}$.
However, if $a_3b_3\neq 0$,
then the displayed equations yield $a_1=a_2$, $b_1=b_2$, and $b_3=-1/a_3$,
whence
$f_1=f_2=a_1b_1$,
$f_3=a_3b_1$,
$f'_3=-a_1/a_3$,
that is,
\[
F=(abx_2x_3+abx_1x_3+bx_1x_2-ax_3^2)\cdot x_4\cdots x_n
\]
after setting
$a:=a_1/a_3$ and $b:=a_3b_1$.
\end{proof}

\subsection{General plan of the proof}

Because of Lemmas~\ref{lemma:linear} and~\ref{lemma:exceptions},
in order to prove Theorem~\ref{thm:form} it remains to show
that if $E=FG$ is any factorisation into non-monomial factors,
then either $F$ or $G$ is the product of a monomial and a linear
factor.
This will be our goal from now on.
Hence we may set the following assumptions, which will make the subsequent
arguments run smoother.

\begin{assumptions}\label{ass}
Let the polynomial $E$ of Theorem~\ref{thm:form} be the product of two polynomials $F$ and $G$, of degrees $r>1$ and $n-r>1$, neither of which is a monomial.
Assume also that $G$ has no non-trivial monomial factor.
\end{assumptions}

Our assumptions on the degrees are allowed because otherwise either $F$ or $G$ would be the desired non-monomial linear factor.
That $G$ has no non-trivial monomial factor can always be achieved
by moving any monomial factor from $G$ to $F$.

\subsection{The case where either $F$ or $G$ is not multilinear}
If $F$ is multilinear but $G$ is not, then we may interchange the roles of $F$ and $G$,
after moving any monomial factor so that Assumptions~\ref{ass} remain satisfied.
Hence assume that $F$ is not multilinear.

Possibly after renumbering the indeterminates, we may assume that $F$ has a term $x_1^2x_2\cdots x_{r-1}$.
Then $G$ is a multilinear polynomial of degree $n-r$ in the
remaining $n-r+1$ indeterminates $x_r,\ldots,x_n$,
otherwise Property~\ref{property} would be contradicted.
Because $G$ has no non-trivial monomial factor according to Assumptions~\ref{ass}, each term
$x_r\cdots x_n/x_i$ with $i\ge r$ appears in $G$ with a nonzero
coefficient.
In turn, Property~\ref{property} implies that any non-multilinear term of $F$ can only involve the indeterminates
$x_1,\ldots,x_{r-1}$.

If some (multilinear) term of $F$ involved at least two of the indeterminates $x_r,\ldots,x_n$,
then because $n-r>1$ that term would share at least two indeterminates with
some term of $G$, and this would contradict Property~\ref{property}.
Therefore, no term of $F$ involves more than one indeterminate from $x_r,\ldots,x_n$.
We have seen earlier that any term of $F$ which involves such indeterminate must be multilinear,
and so altogether $F$ can be written in the form
\[
F=x_1\cdots x_{r-1}\cdot\sum_{i=1}^{n}f_i\,x_i,
\]
which provides us with the desired linear factor.

\subsection{The case where $F$ and $G$ are both multilinear}

Now we may suppose that both $F$ and $G$ are multilinear polynomials.
Because of Property~\ref{property} each term of $G$ can share at most one
indeterminate with each term of $F$.
Any two distinct terms of $F$
must involve together exactly $r+1$ or $r+2$ indeterminates.
In fact, if they did involve more, then
any term of $G$ would involve at least three of them,
and hence it would share at least two indeterminates
with at least one of the two terms of $F$ under consideration,
contradicting Property~\ref{property}.
We deal with those two cases separately.

\subsection{The subcase $r+1$}
Suppose first that $F$ has at least two terms which together
involve $r+1$ indeterminates.
After renumbering the
indeterminates we may assume that two such terms are
$f_r\,x_1\cdots x_r+f_{r+1}\,x_1\cdots x_{r-1}\cdot x_{r+1}$,
with $f_r,f_{r+1}\neq 0$.
Because of Property~\ref{property} each term of $G$ can share at most one indeterminate with each of them, and hence
\begin{equation}\label{eq:G}
G=
\sum_{i=1}^{r+1} g_i\,x_i\cdot x_{r+2}\cdots x_n
+\sum_{j=r+2}^n g'_j\,x_r\cdots x_n/x_j
\end{equation}
for some scalars $g_i,g'_i$.
An alternative choice of notation
would be restricting the former summation range
to $i<r$ and extending the latter to $i\ge r$, provided we set
$g'_r=g_{r+1}$ and $g'_{r+1}=g_r$.
We conveniently allow this double notation in what follows.

Comparing the  coefficients of corresponding monomials on both sides of the
equality $FG=E$ we find, in particular,
\begin{equation}\label{eq:g}
f_rg_i=a_{r+1}b_i
\quad\text{and}\quad
f_{r+1}g_i=a_rb_i
\quad\text{for $i<r$},
\end{equation}
\begin{equation}\label{eq:g'}
f_rg'_j=a_jb_r
\quad\text{and}\quad
f_{r+1}g'_j=a_jb_{r+1}
\quad\text{for $j>r+1$.}
\end{equation}
In fact, because $F$ is multilinear any term in the product $FG$ where $x_i$ appears with exponent
two, for some $i<r$, can only arise in one way,
as the product of the term of $G$ with coefficient $g_i$ and a
uniquely determined term of $F$, necessarily one of the two considered above.
A similar argument applies to any term in the product $FG$ which misses the
indeterminate $x_j$, where $j>r+1$.
Because $F$ is multilinear, Equation~\eqref{eq:G} implies that
$E=FG$ has no term of the form
$x_1\cdots x_n\cdot x_i/x_j$
with $i<r$ and $j>r+1$,
which means $a_jb_i=0$.
Hence either
$b_i=0$ for all $i<r$, or $a_j=0$ for all $j>r+1$.
However, the latter together with Equation~\eqref{eq:g}
yields that $g'_j=0$ for $j>r+1$, whence
$G=x_{r+2}\cdots x_n\cdot\sum_{i=1}^{r+1} g_ix_i$
has a non-trivial monomial factor, against Assumptions~\ref{ass}.
We conclude that $b_i=0$ for $i<r$,
and Equation~\eqref{eq:g} yields $g_i=0$ for $i<r$, and so
$G=\sum_{j=r}^n g'_j\,x_r\cdots x_n/x_j$.

Because $G$ has no non-trivial monomial factor we have
$g'_j\neq 0$ for $j\ge r$.
This implies that no term of $F$ can involve more than one
indeterminate from the set
$\{x_r,\ldots,x_n\}$, and hence
\[
F=x_1\cdots x_{r-1}\cdot\sum_{j=r}^nf_j\,x_j,
\]
providing us with the desired linear factor.

\subsection{The subcase $r+2$}
Now we may assume that each two distinct terms of $F$
involve together exactly $r+2$ indeterminates.
We will deduce a contradiction.
Possibly after renumbering the
indeterminates we may assume that two of the terms of $F$ are (nonzero) scalar multiples of
$x_1x_2\cdot x_5\cdots x_{r+2}$ and $x_3x_4\cdot x_5\cdots x_{r+2}$
(to be appropriately interpreted in case $r=2$).
Because each term of $G$ involves exactly $n-r$ indeterminates, and can share at most one indeterminate with
each of those two terms of $F$ according to Property~\ref{property}, it must involve
all of $x_{r+3},\ldots,x_n$, and exactly one indeterminate from each of the two sets
$\{x_1,x_2\}$ and
$\{x_3,x_4\}$.
Because $G$ has no non-trivial monomial factor we have $r=n-2$, and hence
\[
G=g_{13}x_1x_3+g_{14}x_1x_4+g_{23}x_2x_3+g_{24}x_2x_4
\]
for certain scalars $g_{ij}$.
Again because $G$ has no non-trivial monomial factor we have either $g_{13}g_{24}\neq 0$ or $g_{14}g_{23}\neq 0$.
After possibly exchanging $x_3$ and $x_4$ we may assume the former.

Now according to Property~\ref{property} each term of $F$ can share at most one indeterminate with
each term of $G$, and this leaves only
$x_1x_4\cdot x_5\cdots x_n$ and $x_2x_3\cdot x_5\cdots x_n$
as possible terms of $F$ besides those two assumed from the start.
However, each of these must be excluded because together with one of the initial terms it involves only $n-1=r+1$
indeterminates, rather than $r+2$.
Hence we have
\[
F=(
f_{12}x_1x_2+
f_{34}x_3x_4
)\cdot x_5\cdots x_n,
\]
with $f_{12}f_{34}\neq 0$.
Comparing coefficients of corresponding terms in the equality $FG=E$ we find
$f_{12}g_{13}=a_4b_1$ and
$f_{34}g_{13}=a_2b_3$,
whence $a_2b_1\neq 0$.
This means that $FG$ has a nonzero term $x_1^2x_3x_4\cdot x_5\cdots x_n$,
which is clearly impossible.

This contradiction completes our proof of Theorem~\ref{thm:form}.

\bibliography{References}

\def\cprime{$'$} \def\polhk#1{\setbox0=\hbox{#1}{\ooalign{\hidewidth
  \lower1.5ex\hbox{`}\hidewidth\crcr\unhbox0}}}
\providecommand{\bysame}{\leavevmode\hbox to3em{\hrulefill}\thinspace}
\providecommand{\MR}{\relax\ifhmode\unskip\space\fi MR }
\providecommand{\MRhref}[2]{%
  \href{http://www.ams.org/mathscinet-getitem?mr=#1}{#2}
}
\providecommand{\href}[2]{#2}
\begin{thebibliography}{CDVS09}

\bibitem[CDVS09]{CDVS:AES}
Andrea Caranti, Francesca Dalla~Volta, and Massimiliano Sala, \emph{An
  application of the {O}'{N}an-{S}cott theorem to the group generated by the
  round functions of an {AES}-like cipher}, Des. Codes Cryptogr. \textbf{52}
  (2009), no.~3, 293--301. \MR{2506729 (2010a:94053)}

\bibitem[Csa13]{Csajbok:inverse-closed}
Bence Csajb{\'o}k, \emph{Linear subspaces of finite fields with large
  inverse-closed subsets}, Finite Fields Appl. \textbf{19} (2013), 55--66.
  \MR{2996759}

\bibitem[EOT10]{Tao+:Kakeya}
Jordan~S. Ellenberg, Richard Oberlin, and Terence Tao, \emph{The {K}akeya set
  and maximal conjectures for algebraic varieties over finite fields},
  Mathematika \textbf{56} (2010), no.~1, 1--25. \MR{2604979 (2011c:14066)}

\bibitem[GGSZ06]{GGSZ}
Daniel Goldstein, Robert~M. Guralnick, Lance Small, and Efim Zelmanov,
  \emph{Inversion invariant additive subgroups of division rings}, Pacific J.
  Math. \textbf{227} (2006), no.~2, 287--294. \MR{2263018 (2007i:17041)}

\bibitem[Hua49]{Hua:sfield_properties}
Loo-Keng Hua, \emph{Some properties of a sfield}, Proc. Nat. Acad. Sci. U. S.
  A. \textbf{35} (1949), 533--537. \MR{0031471 (11,155c)}

\bibitem[KLS12]{KLS}
G{\'a}bor Korchm{\'a}ros, Valentino Lanzone, and Angelo Sonnino,
  \emph{Projective {$k$}-arcs and 2-level secret-sharing schemes}, Des. Codes
  Cryptogr. \textbf{64} (2012), no.~1-2, 3--15. \MR{2914398}

\bibitem[LN83]{LN}
Rudolf Lidl and Harald Niederreiter, \emph{Finite fields}, Encyclopedia of
  Mathematics and its Applications, vol.~20, Addison-Wesley Publishing Company
  Advanced Book Program, Reading, MA, 1983, With a foreword by P. M. Cohn.
  \MR{746963 (86c:11106)}

\bibitem[LW54]{Lang-Weil}
Serge Lang and Andr{\'e} Weil, \emph{Number of points of varieties in finite
  fields}, Amer. J. Math. \textbf{76} (1954), 819--827. \MR{0065218 (16,398d)}

\bibitem[Mat]{Mat:caps}
Sandro Mattarei, \emph{A property of the inverse of a subspace of a finite
  field}, {\sf arXiv:1312.1293}, submitted.

\bibitem[Mat07]{Mat:inverse-closed}
\bysame, \emph{Inverse-closed additive subgroups of fields}, Israel J. Math.
  \textbf{159} (2007), 343--347. \MR{2342485 (2008j:12008)}

\end{thebibliography}

\end{document}